\documentclass[12pt,oneside]{amsart}

\usepackage{amssymb,latexsym}

\usepackage{hyperref}
\hypersetup{
  colorlinks   = true, 
  urlcolor     = blue, 
  linkcolor    = blue, 
  citecolor   = red 
}
\usepackage{anysize}
\marginsize{2cm}{2cm}{2cm}{2cm}

\usepackage[pdftex]{graphicx}
\usepackage{subcaption}
\usepackage{mathtools}

\newtheorem{theorem}{Theorem}[section]

\newtheorem{lemma}[theorem]{Lemma}

\theoremstyle{definition}

\theoremstyle{remark}

\numberwithin{equation}{section}

\DeclareMathOperator{\vol}{vol}

\DeclareMathOperator{\conv}{conv}

\DeclareMathOperator{\spn}{span}

\renewcommand{\epsilon}{\varepsilon}
\renewcommand{\phi}{\varphi}
\renewcommand{\kappa}{\varkappa}

\begin{document}

\title{Contractibility of Space of Symplectically Self-polar Convex Bodies}

\author{Mark Berezovik}

\address{Mark Berezovik, School of Mathematical Sciences, Tel Aviv University, Israel,  69978}

\email{m.berezovik@gmail.com}

\thanks{The author was partially supported by ISF grant 769/26 and ISF-NSFC grant 3231/23.}

\begin{abstract}
  Symplectically self-polar bodies have been introduced recently. They appeared in the context of Mahler's conjecture and in the context of symplectic outer billiard dynamics. In this paper we discuss the topology of the space of symplectically self-polar convex bodies. Namely, we show that this space is contractible. 
\end{abstract}


\maketitle

\section{Introduction}\label{sec:introduction}

Symplectically self-polar convex bodies have been recently introduced in~\cite{berezovik2025symplectic}. Let $\omega$ be the standard symplectic form in $\mathbb{R}^{2n}$. We call a convex body $X \subset \mathbb{R}^{2n}$ with the origin in its interior \emph{symplectically self-polar} if $X = X^\omega$, where
\[
   X^\omega = \{y \in \mathbb{R}^{2n} \, | \, \forall x \in X\, \omega(x,y) \leq 1\}
\]
is a symplectic polar transformation of $X$. The latter notion recently also appeared in~\cite{albers2025outer,albers2026symplectic,berezovik2025symplectically,berezovik2026outer,deGosson2024}.

 Initial motivation for the study of this class of bodies was the connection to (symmetric) Mahler's conjecture~\cite{Mahler1939}, which asserts that
\[
  \vol K \cdot \vol K^\circ \geq \frac{4^n}{n!}
\]
for any centrally symmetric convex body $K \subset \mathbb{R}^n$ and its Euclidean polar $K^\circ$. It was shown in~\cite{berezovik2025symplectic} that Mahler's conjecture is equivalent to a conjecture about volumes of symplectically self-polar convex bodies. The latter conjecture states that
\[
  \vol X \geq \frac{2^n}{n!}
\]
for any symplectically self-polar convex body $X \subset \mathbb{R}^{2n}$.

Later, in~\cite{berezovik2026outer}, it was found that this class of bodies is interesting from the dynamical point of view. Namely,  symplectically self-polar convex bodies give non-trivial examples of convex bodies in $\mathbb{R}^{2n}$ which possess invariant hypersurfaces for the symplectic outer billiard map introduced in~\cite{tabachnikov1995dual}. Each such hypersurface consists of 4-periodic centrally symmetric orbits.

Symplectically self-polar convex bodies were also studied in~\cite{berezovik2025symplectically}. In particular, a sequence of symplectically self-polar polytopes $P_n \subset \mathbb{R}^{2n}$ was constructed such that each $P_n$ minimizes the Ekeland--Hofer--Zehnder capacity among symplectically self-polar convex bodies in $\mathbb{R}^{2n}$. Moreover, it seems that each $P_n$ minimizes volume among symplectically self-polar convex bodies in $\mathbb{R}^{2n}$.

However, the topology induced by the Hausdorff distance on this class of bodies has not been discussed before. In this paper, we show that it is trivial.

\begin{theorem}\label{thm:main}
  The space of symplectically self-polar bodies of a given dimension is contractible.
\end{theorem}

\subsection*{Acknowledgments} The author thanks Roman Karasev for useful comments and remarks.

\section{Preliminaries and proof of theorem}\label{sec:prelim}

Before we begin the proof of the theorem, we need to make several preliminary definitions and observations. The symplectic $\ell_2$-sum of convex bodies $X,Y \subset \mathbb{R}^{2n}$, containing the origin in their interior, is defined by
\[
  X \oplus_2 Y = \{(x,y) \in \mathbb{R}^{2n}\times \mathbb{R}^{2n}\, |\, \|x\|_X^2 + \|y\|_Y^2 \leq 1\}.
\]
Here $\|\cdot\|_X, \|\cdot\|_Y$ are norms (not necessarily symmetric) whose unit balls are $X$ and $Y$ respectively. We call this $\ell_2$-sum symplectic in the sense that we consider the space $\mathbb{R}^{2n} \times \mathbb{R}^{2n}$ as a direct sum of two symplectic spaces. Note that in this case $(X \oplus_2 Y)^\omega = X^\omega \oplus_2 Y^\omega$, the proof of this fact is similar to the case of the Euclidean $\ell_2$-sum and the standard polar transformation.

We also need the notion of linear symplectic reduction. Let $X \subset \mathbb{R}^{2n}$ be a convex body with the origin in its interior and $L \subset \mathbb{R}^{2n}$ be an isotropic linear subspace. Consider the corresponding coisotropic $\omega$-orthogonal subspace:
\[
  L^{\perp_{\omega}} = \{y \in \mathbb{R}^{2n}\ |\ \forall x \in L\ \omega (x,y) = 0\}.
\]
Then the symplectic reduction of $X$ along $L$ is the convex body $Y = (X \cap L^{\perp_{\omega}})/L$ in the symplectic space $L^{\perp_{\omega}}/L$. For simplicity, we use the same notation $\omega$ for the symplectic form on $L^{\perp_{\omega}}/L$.

\begin{lemma}[Lemma 4.1 in \cite{berezovik2025symplectic}] \label{lem:reduction}
  Symplectic polar transformation and linear symplectic reduction commute. In particular, linear symplectic reduction of symplectically self-polar convex body is symplectically self-polar.
\end{lemma}
We provide a proof of this statement here for completeness.

\begin{proof}
   Let $X \subset \mathbb{R}^{2n}$ be a convex body with the origin in its interior. Consider its reduction $Y = (X \cap L^{\perp_{\omega}})/L$ along an isotropic subspace $L$. From the definition of symplectic polarity it is easy to see that $Y^\omega \supseteq (X^\omega \cap L^{\perp_{\omega}})/L$.
   
   It remains to prove the reverse inclusion. Let $y^* \in Y^\omega$. Consider the functional $\lambda (\cdot) = \omega(\pi(\cdot),y^*)$ defined on $L^{\perp_{\omega}}$, where $\pi$ is the projection from $L^{\perp_{\omega}}$ to $L^{\perp_{\omega}}/L$. Note that $\lambda|_L = 0$ and $\lambda(x) \leq \|x\|_X$ for every $x \in L^{\perp_{\omega}}$. By the Hann--Banach theorem we can extend $\lambda$ such that $\lambda (x) \leq \|x\|_X$ for every $x \in \mathbb{R}^{2n}$. Since $\omega$ is non-degenerate, there exists unique vector $x^* \in \mathbb{R}^{2n}$ such that $\lambda(\cdot) = \omega(\cdot,x^*)$. This vector belongs to $L^{\perp_{\omega}}$ since $\lambda|_L = 0$ and to $X^\omega$ since $\lambda(\cdot) \leq \|\cdot\|_X$. Moreover, by the definition of $\lambda$ we have $\omega(\cdot,y^*) = \omega(\cdot, \pi(x^*))$ on $L^{\perp_{\omega}}/L$, hence $\pi(x^*) = y^*$. Thus, $Y^\omega \subseteq (X^\omega \cap L^{\perp_{\omega}})/L$.
\end{proof}

\begin{proof}[Proof of Theorem~\ref{thm:main}]
  Let $X, Y \subset \mathbb{R}^{2n}$ be symplectically self-polar convex bodies. Consider their symplectic $\ell_2$-sum $Z = X \oplus_2 Y \subset \mathbb{R}^{2n} \times \mathbb{R}^{2n}$. Note that $Z$ is also symplectically self-polar. Indeed, $Z^\omega = (X\oplus_2 Y)^\omega = X^\omega \oplus_2 Y^\omega = X \oplus_2 Y = Z$.

  Let $\{e_i, f_i\}_{i=1}^n$ be the standard symplectic basis in the first factor of $\mathbb{R}^{2n} \times \mathbb{R}^{2n}$ and $
  \{u_i, v_i\}_{i=1}^n$ be the standard symplectic basis in the second factor.

  Consider the following family of subspaces
  \[
      L_{\varphi} = \spn\{v_i\cdot \cos \varphi - f_i \cdot \sin \varphi\}_{i=1}^n \subset \mathbb{R}^{2n} \times \mathbb{R}^{2n},
  \]
  for $\varphi \in [0,\pi/2]$. Note that every subspace $L_\varphi$ is isotropic and the corresponding symplectic orthogonal subspace can be described as
  \[
      L_{\varphi}^{\perp_\omega} = L_\varphi \oplus \spn\{e_i \cdot \cos \varphi + u_i \cdot \sin \varphi\}_{i=1}^n \oplus \spn\{f_i \cdot \cos \varphi + v_i \cdot \sin \varphi\}_{i=1}^n.
  \]
  Note that images of vectors $\{e_i \cdot \cos \varphi + u_i \cdot \sin \varphi, f_i \cdot \cos \varphi + v_i \cdot \sin \varphi\}_{i=1}^n$ form symplectic basis in $L_{\varphi}^{\perp_\omega}/L_\varphi$. Using this choice of basis in each $L_{\varphi}^{\perp_\omega}/L_\varphi$ we identify this space with $\mathbb{R}^{2n}$ with the standard symplectic form. Let us denote this identification by $H_\varphi \colon L_{\varphi}^{\perp_\omega}/L_\varphi \to \mathbb{R}^{2n}$.

  Consider the family $M_\varphi = H_\varphi((Z \cap L_{\varphi}^{\perp_\omega}) / L_\varphi) \subset \mathbb{R}^{2n}$. From Lemma~\ref{lem:reduction} it follows that each $M_{\varphi}$ is symplectically self-polar. 

  Note that $M_0 = X$. Indeed, $L_0 = \spn\{v_i\}_{i=1}^n$ and $L_{0}^{\perp_\omega} = L_0 \oplus \spn\{e_i,f_i\}_{i=1}^n$. Hence, $x \in M_0$ if and only if there exists $y \in L_0$ such that $\|x\|_X^2 + \|y\|_Y^2 \leq 1$. Therefore, $x \in M_0$ if and only if $\|x\|_X^2 \leq 1$. Thus, $M_0 = X$. In the similar way one can show that $M_{\pi/2} = Y$. 

  Note also that $M_{\varphi}$ depends continuously from $\varphi$ and the choice of $X$ and $Y$. Thus, if we fix $Y$, choose $X$ arbitrary, and change the parameter $\varphi$ from $0$ to $\pi/2$, we get homotopy of the space which shrinks it to one point.
\end{proof}

\section{Examples}\label{sec:examples}

In this section we show how transformation from one symplectically self-polar body to another, described in the proof of Theorem~\ref{thm:main}, looks like in dimension two for some pairs $X$ and $Y$. 

\begin{center}
  \textbf{Hexagon - Hexagon} 
\end{center}
Consider two different symplectically self-polar hexagons $X$ and $Y$ in $\mathbb{R}^2$ such that
\begin{align*}
  \|(q,p)\|_X = \max\{|q|, |p|, |q-p|\},\\
  \|(q,p)\|_Y = \max\{|q|, |p|, |q+p|\}.
\end{align*}
Note that symplectically self-polar hexagons play special role among all planar symplectically self-polar convex bodies. Namely, symplectically self-polar hexagon is unique planar symplectically self-polar convex body of minimal volume up to linear symplectomorphism~\cite[Theorem 4.1]{berezovik2025symplectically}. 

Then following the proof of Theorem~\ref{thm:main} the norm of $Z = X \oplus_2 Y$ by the definition has the form
\[
  \|(q_1,p_1,q_2,p_2)\|_Z^2 = \|(q_1,p_1)\|_X^2 + \|(q_2,p_2)\|_Y^2.
\]

Let us consider the following function
\begin{align*}
  g_{\varphi}(q,p,t) &= \|q\cdot (e \cdot \cos \varphi + u \cdot \sin \varphi) + p\cdot (f \cdot \cos \varphi + v \cdot \sin \varphi) + t \cdot (v \cdot \cos \varphi - f \cdot \sin \varphi) \|_Z^2 =\\
  &= \|(q\cdot \cos \varphi,\, p \cdot \cos \varphi - t \cdot \sin \varphi,\,  q \cdot \sin \varphi, t \cdot \cos \varphi + p \cdot \sin \varphi)\|_Z^2
\end{align*}

Then, by the definition of $M_\varphi$, one has
\[
  M_\varphi = \{(q,p) \in \mathbb{R}^2\, |\, \exists t \in \mathbb{R}: g_{\varphi}(q,p,t) \leq 1\}.
\]
Consider the following family of symplectically self-polar hexagons:
\[
  P_\varphi = \conv \{\pm(0,1), \pm(-1,1-\cos^2\varphi), \pm(-1,-1+\sin^2\varphi)\}.
\]
We are going to show that $M_\varphi = P_\varphi$. Let us first show that points $(0,1), (-1,1-\cos^2\varphi)$, and $(-1,-1+\sin^2\varphi)$ belong to $M_\varphi$. Indeed, $g_{\varphi} (0,1,0) = g_{\varphi} (-1,1-\cos^2 \varphi,\sin(2\varphi)/2) = g_{\varphi} (-1,-1+\sin^2 \varphi,\sin(2\varphi)/2) = 1$. Together with the fact that $P_\varphi$ and $M_\varphi$ are centrally symmetric it follows that $P_\varphi \subseteq M_\varphi$. Note that symplectic polar transformation changes the direction of inclusion, therefore $ M_\varphi^\omega \subseteq P_\varphi^\omega$. However, $P_\varphi = P_\varphi^\omega$ and $M_\varphi = M_\varphi^\omega$. Thus, $M_\varphi = P_\varphi$.

\begin{center}
  \textbf{Ball - Hexagon} 
\end{center}
This time let us consider the Euclidean unit ball and the hexagon such that 
\begin{align*}
  \|(q,p)\|_X &= \sqrt{q^2 + p^2},\\
  \|(q,p)\|_Y &= \max\{|q|,|p|, |q-p|\}.
\end{align*}
In contrast to hexagon, the Euclidean unit ball is unique planar symplectically self-polar convex body of maximal volume up to linear symplectomorphism. This follows directly from the equality case in the Blaschke--Santaló inequality~\cite{Santalo1949}.

Similar to the previous case consider the function
\[
  g_{\varphi}(q,p,t) = \|(q\cdot \cos \varphi,\, p \cdot \cos \varphi - t \cdot \sin \varphi)\|_X^2 + \|(q \cdot \sin \varphi, t \cdot \cos \varphi + p \cdot \sin \varphi)\|_Y^2.
\]
Then
\[
  M_\varphi = \{(q,p) \in \mathbb{R}^2| \exists t \in \mathbb{R} : g_{\varphi}(q,p,t) \leq 1\}.
\]
Consider family $K_\varphi$ of symplectically self-polar convex bodies of the form, see Figure~\ref{fig:bodies_K_phi},

\begin{figure}[!ht]
    \centering 
    \includegraphics[width=0.6\textwidth]{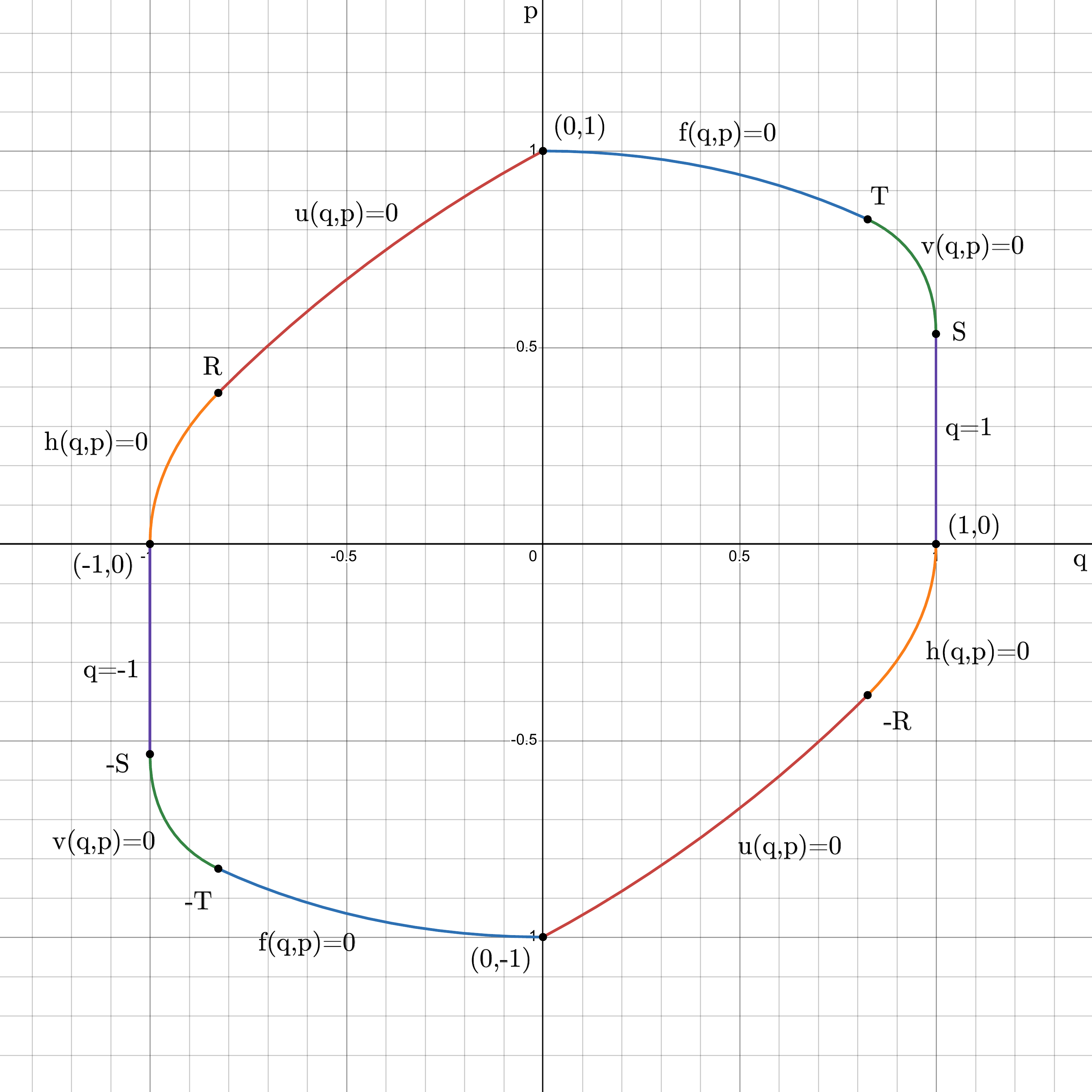}
    \caption{Boundary of the body $K_\varphi$.}
    \label{fig:bodies_K_phi}
  \end{figure}

where
\begin{align*}
   R &= \left(- \frac{1}{\sqrt{1+ \cos^2 \varphi}}, \frac{\cos^2\varphi}{\sqrt{1+ \cos^2\varphi}}\right),\\
  T &= \left(\frac{1}{\sqrt{1 + \cos^2 \varphi}},\frac{1}{\sqrt{1 + \cos^2 \varphi}}\right),\\
  S &= (1,\sin^2 \varphi).
\end{align*}
and
\begin{align*}
  h(q,p) &= \cos^2 \varphi \cdot q^2 + p^2 - \cos^2 \varphi,\\
  u(q,p) &= (7 + \cos 4\varphi)\cdot q^2 + 8\cdot (\cos2\varphi - 1) \cdot qp + 8p^2 -8,\\
    f(q,p) &= \cos^2 \varphi \cdot q^2 + p^2 - 1,\\
  v(q,p) &= (7+ \cos 4 \varphi) \cdot q^2 + 8(\cos 2\varphi - 1) \cdot qp + 8p^2 - 4 (1+ \cos 2\varphi).
\end{align*}
One can prove that $K_{\varphi}$ is symplectically self-polar for every $\phi$. Similar to the previous case let us prove that $K_{\varphi} \subseteq M_\varphi$, then it will follow that $K_{\varphi} = M_{\varphi}$. It is sufficient to show that the boundary of $K_{\varphi}$ for $p \geq 0$ is contained in $M_{\varphi}$. For the upper boundary the following holds:
\begin{align*}
  p_h(q) &= \cos \varphi \cdot \sqrt{1-q^2}, &-1 \leq q \leq -\frac{1}{\sqrt{1+ \cos^2 \varphi}},\\
  p_u(q) &= q\cdot \sin^2 \varphi + \sqrt{1-q^2 \cos^2 \varphi}, &-\frac{1}{\sqrt{1+\cos^2\varphi}} \leq q \leq 0,\\
  p_f(q) &= \sqrt{1- q^2 \cdot \cos^2\varphi}, &0 \leq q \leq \frac{1}{\sqrt{1 + \cos^2 \varphi}},\\
  p_v(q) &= q \cdot \sin^2 \varphi + \cos \varphi \sqrt{1-q^2}, &\frac{1}{\sqrt{1+ \cos^2 \varphi}} \leq q \leq 1.
\end{align*}
Consider the following collection of functions:
\begin{align*}
  t_h(q) &= -\sin \varphi \cdot \sqrt{1-q^2},\\
  t_u(q) &= \cos \varphi \cdot \sin \varphi \cdot q,\\
  t_v(q) &= \cos \varphi \cdot \sin \varphi \cdot q - \sin \varphi \cdot \sqrt{1-q^2},\\
  t_{\text{vert}}(p) &= \frac{p}{\tan \varphi}.
\end{align*}
Then
\begin{align*}
  g_{\varphi} (q,p_h(q),t_h(q)) &= 1, &-1 \leq q \leq -\frac{1}{\sqrt{1+\cos^2 \varphi}},\\
  g_{\varphi} (q,p_u(q),t_u(q)) &= 1, &-\frac{1}{\sqrt{1+\cos^2 \varphi}} \leq q \leq 0,\\
  g_{\varphi} (q, p_f(q),0) &= 1, &0 \leq q \leq \frac{1}{\sqrt{1+ \cos^2 \varphi}},\\
  g_{\varphi} (q, p_v(q),t_v(q)) &= 1, &\frac{1}{\sqrt{1+ \cos^2 \varphi}} \leq q \leq 1,\\
  g_{\varphi} (1,p,t_{\text{vert}}(p)) &= 1, &0\leq p \leq \sin^2 \varphi.
\end{align*}
This finishes the proof.
\bibliography{bibliography}
\bibliographystyle{abbrv}
\end{document}